\documentclass[12pt]{amsart}

\usepackage[cp850]{inputenc}

\usepackage{graphicx}

\usepackage{psfrag}

\usepackage{amsxtra,amssymb}

\usepackage{amssymb,amsmath}
\usepackage{geometry}    
\usepackage{mathrsfs}

\usepackage{vmargin}
\setmargrb{1in}{1in}{1in}{1in}

\newtheorem{theorem}{Theorem}[section]

\newtheorem{proposition}[theorem]{Proposition}
\newtheorem{corollary}[theorem]{Corollary}
\theoremstyle{definition}
\newtheorem{definition}{Definition}
\newtheorem{remark}{Remark}

\newtheorem{conjecture}[theorem]{Conjecture}

\begin{document}

\title[Notes on the regular graph]
{Some notes on the regular graph defined by Schmidt and Summerer and uniform approximation}

\author{Johannes Schleischitz}
                            
\thanks{Supported by FWF grant P24828,\\
	Institute of Mathematics, Department of Integrative Biology, BOKU Wien, 1180, Vienna, Austria \\
	johannes.schleischitz@boku.ac.at}

\begin{abstract} 
Within the study of parametric geometry of numbers  
W. Schmidt and L. Summerer introduced so-called regular graphs. 
Roughly speaking the successive minima functions for the classical
simultaneous Diophantine approximation problem have a very special pattern
if the vector $\underline{\zeta}$ induces a regular graph. The regular graph
is in particular of interest due to a conjecture by Schmidt and Summerer concerning classic approximation
constants. This paper aims to provide several new results on the behavior of the
successive minima functions for the regular graph. Moreover, we improve the best known
upper bounds for the classic approximation constants $\widehat{w}_{n}(\zeta)$, 
provided that the Schmidt-Summerer conjecture is true.   
\end{abstract}

\maketitle

{\footnotesize{Math subject classification: 11J13, 11J25, 11J82 \\
key words: successive minima, lattices, regular graph, uniform Diophantine approximation}}

\section{Introduction} \label{sektion1} 

\subsection{Outline}
This paper aims on the one hand to give a better understanding of the regular graph defined by Schmidt and Summerer,
and on the other hand to establish a connection to the uniform approximation constants $\widehat{w}_{n}$. 
Theorem~\ref{bergzwerg} and Theorem~\ref{moshthm} can be considered the main results concerning the first, 
Theorem~\ref{reggraph} the main result for the latter topic. 
In Section~\ref{sscon} we will define the regular graph and explain 
its significance for simultaneous Diophantine approximation.
We recommend the reader to look at the illustrations of combined graphs and in particular the 
regular graph in~\cite[page 90]{sums}, another sketch adopted from~\cite[page 72]{j1}
is visible in Section~\ref{sscon}.
See also~\cite{j1} for Matlab plots of the combined graph
for special choices of real vectors. Finally in Section~\ref{andere} we discuss 
the consequences of another reasonable conjecture to uniform approximation.

\subsection{Geometry of numbers} \label{param}

We start with a classical problem of simultaneous approximation.
Assume $\underline{\zeta}=(\zeta_{1},\ldots,\zeta_{n})$ in $\mathbb{R}^{n}$ is given.
For $1\leq j\leq n+1$ let $\lambda_{n,j}=\lambda_{n,j}(\underline{\zeta})$ be the supremum of real $\nu$ 
for which there are arbitrarily large $X$ such that the system
\begin{equation} \label{eq:systematisch}
\vert x\vert\leq X, \qquad \max_{1\leq j\leq k} \vert \zeta_{j}x-y_{j}\vert \leq X^{-\nu}
\end{equation}
has $j$ linearly independent solution vectors $(x,y_{1},\ldots,y_{n})$ in $\mathbb{Z}^{n+1}$. 
Moreover let $\widehat{\lambda}_{n,j}=\widehat{\lambda}_{n,j}(\underline{\zeta})$ be
the supremum of $\nu$ such that the system \eqref{eq:systematisch}
has $j$ linearly independent integer vector solutions $(x,y_{1},\ldots,y_{n})$ 
for all large $X$. 
For $\lambda_{n,1}$ we will also simply write $\lambda_{n}$, and similarly $\widehat{\lambda}_{n}$
for $\widehat{\lambda}_{n,1}$. For all $\underline{\zeta}\in{\mathbb{R}^{n}}$, 
Minkowski's first lattice point theorem (or Dirichlet's Theorem)
implies the estimates
\begin{equation} \label{eq:dir1}
\lambda_{n}\geq \widehat{\lambda}_{n}\geq \frac{1}{n}.
\end{equation}
More generally it can be shown that
\begin{equation}
\frac{1}{n}\leq \lambda_{n}\leq \infty, \qquad
\frac{1}{n}\leq \lambda_{n,2}\leq 1, \qquad
0\leq \lambda_{n,j}\leq \frac{1}{j-1}\quad (1\leq j \leq n+1),   \label{eq:1}  
\end{equation}
and similarly
\begin{equation}
\frac{1}{n}\leq \widehat{\lambda}_{n}\leq 1, \qquad
0\leq \widehat{\lambda}_{n,j}\leq \frac{1}{j} \quad (2\leq j\leq n), \qquad
0\leq \widehat{\lambda}_{n,n+1}\leq \frac{1}{n}.       \label{eq:6}  
\end{equation}
See \cite[(14)-(18)]{j1}. Moreover $\lambda_{n,j}\geq \widehat{\lambda}_{n,j-1}$ holds for $2\leq j\leq n+1$ as pointed out in~\cite{ss}. 

Schmidt and Summerer investigated
a parametric version of the simultaneous approximation problem above~\cite{ss},~\cite{ssch}.
We will now introduce some concepts and results of the evolved parametric geometry 
of numbers from~\cite{ss}. 
Our notation will partially deviate from~\cite{ss}
for technical reasons. 
Keep $n\geq 1$ an integer
and $\underline{\zeta}=(\zeta_{1},\ldots,\zeta_{n})$ a fixed vector in $\mathbb{R}^{n}$.
For any parameter $Q>1$ and any $1\leq j\leq n+1$, consider the largest number $\nu$ such that
\[
\vert x\vert \leq Q^ {1+\nu}, \qquad  
\max_{1\leq j\leq k} \vert \zeta_{j}x-y_{j}\vert \leq Q^ {-1/n+\nu}, 
\] 
has $j$ linearly independent integral solution vectors $(x,y_{1},\ldots,y_{n})\in{\mathbb{Z}^{n+1}}$.
Denote by $\psi_{n,j}(Q)$ this value. Dirichlet's Theorem
yields $\psi_{n,1}(Q)<0$ for all $Q>1$.
Further let
 \[
 \underline{\psi}_{n,j}=\liminf_{Q\to\infty} \psi_{j}(Q), \qquad
 \overline{\psi}_{n,j}=\limsup_{Q\to\infty} \psi_{j}(Q).
 \]
It is not hard to see that
  \[
  -1\leq \psi_{n,1}(Q)\leq \psi_{n,2}(Q)\leq \cdots \leq \psi_{n,n+1}(Q)\leq \frac{1}{n}, \qquad Q>1,
 \]
and in particular
  \[
  -1\leq \underline{\psi}_{n,j}\leq \overline{\psi}_{n,j}\leq\frac{1}{n}, \qquad 1\leq j\leq n+1.
  \]
For $q= \log Q$ consider the derived functions
\begin{equation} \label{eq:lbedingung}
L_{n,j}(q)= q\psi_{n,j}(Q), \qquad 1\leq j\leq n+1.
\end{equation}
They have the nice property of being piecewise linear with slope among $\{-1,1/n\}$. 
The functions $\psi_{n,j}$ and the derived $L_{n,j}$ can alternatively 
be defined via a classical successive minima problem of a parametrized family 
of convex bodies with respect to a lattice. 
For the details see~\cite{ss}. 
A crucial observation from this point of view is that 
Minkowki's second theorem yields pointed out in~\cite{ss} is that the sum
of $L_{n,j}$ over $j$ is uniformly bounded by absolute value for $q>0$.
The connection between the constants $\lambda_{n,j}$ and the functions
$\psi_{n,j}$ is given by the formula
\[
(1+\lambda_{n,j})(1+\underline{\psi}_{n,j})=
(1+\widehat{\lambda}_{n,j})(1+\overline{\psi}_{n,j})=\frac{n+1}{n}, \qquad 1\leq j\leq n+1.
\]
This was pointed out in \cite[(13)]{j1}, which generalized \cite[Theorem~1.4]{ss}.
In particular, for $1\leq j\leq n+1$, we have the equivalences
\begin{equation} \label{eq:aequi}
\underline{\psi}_{n,j}<0 \; \Longleftrightarrow \; \lambda_{n,j}>\frac{1}{n}, \qquad 
\overline{\psi}_{n,j}<0 \; \Longleftrightarrow \; \widehat{\lambda}_{n,j}>\frac{1}{n}.
\end{equation}

We now briefly introduce the dual problem studied in~\cite{ss} as well. 
Define the classic approximation constant $w_{n,j}$ and $\widehat{w}_{n,j}$ respectively as the
supremum of $\nu$ such that the system
\[
\max\{\vert x\vert, \vert y_{1}\vert,\ldots,\vert y_{n}\vert\} \leq X, 
\qquad \vert x+\zeta_{1}y_{1}+\cdots+\zeta_{n}y_{n}\vert \leq X^{-\nu},
\]
has $j$ linearly independent integer vector solutions for 
arbitrarily large $X$ and all large $X$,
respectively. Again we also write $w_{n}$ instead of $w_{n,1}$ and $\widehat{w}_{n}$
instead of $\widehat{w}_{n,1}$. 
In this context Minkowski's first lattice point theorem 
(or Dirichlet's Theorem) implies
\begin{equation} \label{eq:dir2}
w_{n}\geq \widehat{w}_{n}\geq n.
\end{equation}
As already mentioned in~\cite[(1.24)]{j2}, it can be shown that
\begin{equation} \label{eq:platz}
w_{n,j}=\frac{1}{\widehat{\lambda}_{n,n+2-j}}, \qquad \lambda_{n,j}=\frac{1}{\widehat{w}_{n,n+2-j}},
\qquad\qquad 1\leq j\leq n+1.
\end{equation} 
Together with the bounds in \eqref{eq:1} and \eqref{eq:6},
for the spectra of the exponents we obtain 
\begin{equation}
n\leq w_{n}\leq \infty, \qquad n+2-j\leq w_{n,j}\leq \infty,  \quad 2\leq j\leq n, 
\qquad 1\leq w_{n,n+1}\leq n,    \label{eq:11} 
\end{equation}
such as
\begin{equation}
n+1-j\leq \widehat{w}_{n,j}\leq \infty, \quad 1\leq j\leq n-1, \qquad
1\leq \widehat{w}_{n,n}\leq n,  \qquad
0\leq \widehat{w}_{n,n+1}\leq n.       \label{eq:16}  
\end{equation}
Schmidt and Summerer studied a
parametric version of the linear form problem as well in~\cite{ss}, however the
above classic exponents will suffice for our purposes.

\subsection{The regular graph and the Schmidt-Summerer Conjecture} \label{sscon}

For fixed $n\geq 1$ and a parameter $\rho\in{[1,\infty]}$ in~\cite{sums} Schmidt and Summerer define what is called the regular graph.
This geometrically describes a special pattern of the combined graph of the successive minima functions
$L_{n,j}(q)=L_{n,j}(\log Q)$ from Section~\ref{param}. 
We refer to \cite[page 90]{sums} for an idealized 
illustrations of the functions $L_{n,j}(q)$ for 
the regular graph connected to approximation of 
three numbers , i.e. $n=3$ in our notation. Figure~1 below
depicts a sketch for $n=2$, which was already presented in \cite[page~72]{j1}. 
The solid lines depic the graphs of the functions $L_{2,1},L_{2,2},L_{2,3}$
whereas the dotted lines correspond to the quantities 
$\underline{\psi}_{2,j}, \overline{\psi}_{2,j}$ for $1\leq j\leq 3$.
Notice that $\underline{\psi}_{2,j+1}=\overline{\psi}_{2,j}$ in the regular graph.

\begin{figure}[h!]
	\centering
	\scalebox{0.6}{\centering
		\includegraphics{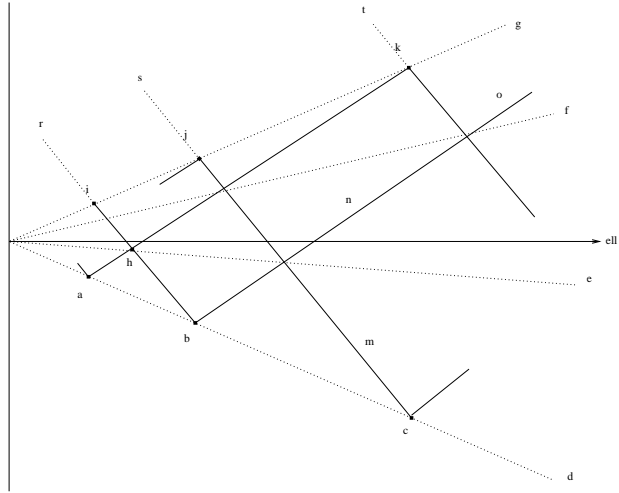}}
	\caption{Sketch of the regular graph for $n=2$}
\end{figure}   

Roughly speaking, the integers $(x_{k})_{k\geq 1}$ that induce 
a not too short falling period of all $L_{n,j}(q)$, 
coincide for all $1\leq j\leq n+1$ and have the additional property that 
the logarithmic quotients $\log x_{k+1}/\log x_{k}$
tend to $\lambda_{n}/\widehat{\lambda}_{n}$. 
An immediate consequence
already mentioned in~\cite[Section~3]{j1} is that all quotients
$\lambda_{n,j}/\lambda_{n,j+1}=\lambda_{n,j}/\widehat{\lambda}_{n,j}$
coincide for $1\leq j\leq n+1$. That is
\begin{equation} \label{eq:quot}
\frac{\lambda_{n}}{\lambda_{n,2}}=\frac{\lambda_{n,2}}{\lambda_{n,3}}=\cdots=\frac{\lambda_{n,n+1}}{\lambda_{n,n+2}}
=\frac{\widehat{\lambda}_{n}}{\widehat{\lambda}_{n,2}}=\cdots=\frac{\widehat{\lambda}_{n,n}}{\widehat{\lambda}_{n,n+1}},
\end{equation}
where we have put $\lambda_{n,n+2}:= \widehat{\lambda}_{n,n+1}$,
which shall remain for the sequel.
Moreover it is obvious from its definition that the regular graph satisfies
\begin{equation} \label{eq:galeich}
\lambda_{n,j}= \widehat{\lambda}_{n,j-1}, \qquad 2\leq j\leq n+2.
\end{equation}
%
%
%
In view of \eqref{eq:quot} and \eqref{eq:galeich},
all $\lambda_{n,j}, \widehat{\lambda}_{n,j}$ are determined by one 
real parameter $\lambda\geq 1/n$. According to
\eqref{eq:platz}, this applies to all exponents $w_{n,j}$ and $\widehat{w}_{n,j}$ as well.
The parameter $\rho\in{[1,\infty]}$ in Schmidt-Summerer notation coincides with the value
$\lambda_{n}/\widehat{\lambda}_{n}$ in \eqref{eq:quot}.
We will use a different 
parametrization. We consider the equivalent situation that the constant 
$\lambda_{n}$ is prescribed in the interval $[1/n,\infty]$.
Any such choice again uniquely determines a regular graph in dimension $n$ and vice versa. 
Thus we have the assignment
\begin{equation} \label{eq:zuordnung}
(n,\lambda)\to (\lambda_{n},\lambda_{n,2},\ldots,\lambda_{n,n+1},\lambda_{n,n+2}), \qquad 
\lambda\in{[1/n,\infty]},
\end{equation}
where $\lambda_{n}=\lambda$. 
We call the graph arising from \eqref{eq:zuordnung}
{\em the regular graph in dimension $n$ with parameter $\lambda$}.
For $n=2$ the graphs of the functions $\lambda_{2,j}$ are illustrated in Figure~2 below. 

\begin{figure}[h!]
	\centering
	\includegraphics[width=0.5\textwidth,natwidth=650,natheight=650]{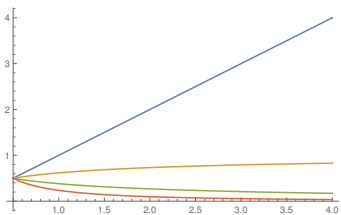}
	\caption{The functions $\lambda_{2,1}(\lambda),\lambda_{2,2}(\lambda),\lambda_{2,3}(\lambda),\lambda_{2,4}(\lambda)$
		in the interval $\lambda\in[1/2,4]$}
\end{figure}
It is rather obvious and
will follow from \eqref{eq:impgl} in Section~\ref{graph} that the right hand side 
in \eqref{eq:zuordnung} depends continuously on $\lambda$.
In view of \eqref{eq:galeich}, the assignment \eqref{eq:zuordnung} 
contains the entire information on all exponents
$\lambda_{n,j}, \widehat{\lambda}_{n,j}$. 
We will also write $\lambda_{n,j}(\lambda)$ and $\widehat{\lambda}_{n,j}(\lambda)$
for the quantity $\lambda_{n,j}$ and $\widehat{\lambda}_{n,j}$ respectively
in the regular graph in dimension $n$ and parameter $\lambda$.
It is worth noting that for $\lambda_{n}=\lambda=1/n$ all constants in \eqref{eq:zuordnung} take the value $1/n$, which a very general elementary consequence of Minkoski's second theorem. 
Moreover, in the
other degenerate case of the regular graph $\lambda=\infty$, it is not hard to see that
$\lambda_{n,2}(\infty)=1$ and $\lambda_{n,j}(\infty)=0$ holds for $3\leq j\leq n+2$, see also 
Proposition~\ref{lagegrund} below. Roy~\cite{roy5} proved that
for any pair $(n,\lambda)$ as in \eqref{eq:zuordnung}, there exist $\mathbb{Q}$-linearly independent vectors $\underline{\zeta}$ (together  with $\{1\}$) that induce the corresponding 
regular graph. The existence of the regular graph
for the special ''degenerate'' case $\lambda=\infty$ had already been constructively 
proved before by the author~\cite[Theorem~4]{j1}.

The importance of the regular graph stems in particular from a conjecture
by Schmidt and Summerer~\cite{sums}. It suggests that the regular graph with assigment
\eqref{eq:zuordnung}  maximizes the value $\widehat{\lambda}_{n}$ 
among all $\underline{\zeta}$ that are $\mathbb{Q}$ linearly independent with $1$ and
share the prescribed value $\lambda_{n}(\underline{\zeta})=\lambda$. A dual version
of the conjecture states that $\widehat{w}_{n}$ is maximized for given value of $w_{n}$
in the regular graph as well. For convenience we introduce some notation.

\begin{definition} \label{lator}
Let $\phi_{n}$ be the function that expresses $\widehat{w}_{n}$ in terms of $w_{n}\in{[n,\infty]}$
and $\vartheta_{n}$ the function that expresses the
value $\widehat{\lambda}_{n}$ in terms of $\lambda_{n}\in{[1/n,\infty]}$ in the regular graph. 
\end{definition}

Note that $\vartheta_{n}(\lambda)$ coincides with $\widehat{\lambda}_{n}(\lambda)=\lambda_{n,2}(\lambda)$
defined above.
The Schmidt-Summerer Conjecture can now be stated in the following way.

\begin{conjecture}[Schmidt, Summerer]   \label{schmsum}
For any positive integer $n$ and every $\underline{\zeta}\in{\mathbb{R}^{n}}$ which is $\mathbb{Q}$-linearly
independent together with $\{1\}$, we have
$\widehat{w}_{n}(\underline{\zeta})\leq \phi_{n}(w_{n}(\underline{\zeta}))$ and 
$\widehat{\lambda}_{n}(\underline{\zeta})\leq \vartheta_{n}(\lambda_{n}(\underline{\zeta}))$.
In particular for any real transcendental $\zeta$ and all $n\geq 1$ we 
have $\widehat{w}_{n}(\zeta)\leq \phi_{n}(w_{n}(\zeta))$ and 
$\widehat{\lambda}_{n}(\zeta)\leq \vartheta_{n}(\lambda_{n}(\zeta))$.  
\end{conjecture}

For $n\in\{2,3\}$ Schmidt and Summerer settled Conjecture~\ref{schmsum} in~\cite{ssch} and \cite{sums},
see also Moshchevitin~\cite{germosh}. For $n\geq 4$ it is open.
As mentioned above equality holds for suitable $\underline{\zeta}$,
so Conjecture~\ref{schmsum} would lead to sharp bounds.

\section{Structural study of the regular graph} \label{newregular}

\subsection{Fixed $\lambda$} \label{fixla}
In this short section let $\lambda>0$ be given.
We investigate constants $\lambda_{n,j}$ in the regular graph for prescribed value $\lambda_{n}=\lambda$ 
in dependence of $n$, for which obviously it is necessary and sufficient to assume $n\geq \lceil \lambda^{-1}\rceil$.
Recall the notation $\lambda_{n,j}(\lambda)$ and $\widehat{\lambda}_{n,j}(\lambda)$
for the constants $\lambda_{n,j}, \widehat{\lambda}_{n,j}$ obtained in the regular
graph in dimension $n$ and the parameter $\lambda_{n,1}=\lambda_{n}=\lambda$. 
Our first result shows
roughly speaking that for fixed $\lambda_{n}=\lambda$, the remaining constants $\lambda_{n,j}(\lambda)$ 
for fixed $j\geq 2$ are decreasing as the dimension $n$ increases.

\begin{proposition} \label{propodil}
Let $\lambda>0$ be fixed and $n_{1}>n_{2}\geq j-1\geq 1$ be integers such that $n_{2}\geq \lceil\lambda^{-1}\rceil$.
Then the constants $\lambda_{n_{i},j}(\lambda), i\in{\{1,2\}}$ in the regular graph in dimensions $n_{1}$ 
and $n_{2}$ respectively and parameter $\lambda$ are well-defined and satisfy 
$\lambda_{n_{1},j}(\lambda)<\lambda_{n_{2},j}(\lambda)$.
\end{proposition}

\begin{remark} \label{markenreh}
The proposition can be used to obtain the following.
Consider the regular graphs in some fixed dimension $n\geq 2$ 
and let the parameter $\lambda$ tend to infinity. Then we have the asymptotic behaviour
\begin{equation} \label{eq:equatin}
\lim_{\lambda\to\infty} \lambda+1-\frac{\lambda}{\widehat{\lambda}_{n}(\lambda)}=0,
\end{equation}
with $\lambda_{n}=\lambda$ and $\widehat{\lambda}_{n}(\lambda)$ as in Section~\ref{sscon}.
The formula \eqref{eq:equatin}
was remarked but not proved in \cite{j1}. Observe that \eqref{eq:equatin}
in particular yields
\[
\lim_{\lambda\to\infty} \widehat{\lambda}_{n}(\lambda)=\lim_{\lambda\to\infty} \lambda_{n,2}(\lambda)=1.
\]
This property can be roughly seen in Figure~1.
\end{remark}

\begin{corollary} \label{korola}
Let $j\geq 2$ be an integer and $\lambda>0$ a fixed parameter. 
Consider the regular graphs in all dimensions 
$n\geq \lceil \lambda^{-1}\rceil$ with $\lambda_{n}=\lambda$ as in \eqref{eq:zuordnung}, which are well-defined.
Then we have 
\[
\lambda_{n,j}(\lambda)\geq \frac{\lambda}{(1+\lambda)^{j-1}}
\]
and the asymptotic behavior
\[
\lim_{n\to\infty} \widehat{\lambda}_{n,j-1}(\lambda)=\lim_{n\to\infty} \lambda_{n,j}(\lambda)
=\frac{\lambda}{(1+\lambda)^{j-1}}.
\]
\end{corollary}

\subsection{Fixed $n$ and Schmidt's conjecture} \label{s32}

First we state a (by now settled) conjecture of W. Schmidt.
Recall the simultaneous approximation problem from Section~\ref{param}
can be interpreted as a successive minima problem of a parametrized family of
convex bodies with respect to a lattice. Schmidt conjectured
that for any integers $1\leq T\leq n-1$ there exist 
vectors $\underline{\zeta}$ that are
$\mathbb{Q}$ linearly independent together with $\{1\}$, and for which
the corresponding $T$-th successive minimum tends to $0$ whereas the $(T+2)$-nd
tends to infinity. 
In the language of Section~\ref{param} 
this means precisely that
the function $L_{n,T}(q)$ tends to $-\infty$ whereas $L_{n,T+2}(q)$
tends to $+\infty$ as $q\to\infty$.
For convenience we introduce some notation.

\begin{definition}
Let $n,T$ be integers with $1\leq T\leq n-1$.
We say $\underline{\zeta}\in{\mathbb{R}^{n}}$ satisfies Schmidt's property for $(n,T)$
if $\underline{\zeta}$ is $\mathbb{Q}$-linearly independent together with $\{1\}$ and
the induced functions $L_{n,j}$ from Section~\ref{param} satisfy
$\lim_{q\to\infty} L_{n,T}(q)=-\infty$ and $\lim_{q\to\infty} L_{n,T+2}(q)=\infty$.
\end{definition}

So Schmidt's conjecture claims that for any reasonable pair $(n,T)$, 
the set of vectors that satisfy Schmidt's property is non-empty.
The conjecture was proved
in a complicated non-constructive way by Moshchevitin~\cite{mosh2}. 
In case of $T$ not too close to $n$,
where the condition $T<n/\log n$ is sufficient, it was reproved constructively in~\cite{j1}. 
We should remark that the modified Schmidt property for 
$L_{n,T}$ and $L_{n,T+1}$ instead of $L_{n,T}$ and $L_{n,T+2}$ 
cannot be satisfied if $\underline{\zeta}$ is $\mathbb{Q}$-linearly independent together
with $\{1\}$. Indeed it must fail since then $L_{n,j}(q)=L_{n,j+1}(q)$
has arbitrarily large solutions $q$ for any $1\leq j\leq n$, see~\cite[Theorem~1.1]{ss}. 
On the other hand, if one drops the linear independence condition, the conjecture would 
be true as well by a rather easy argument, as carried out in~\cite{mosh2}.

By \eqref{eq:lbedingung}, a sufficient condition for a vector to satisfy
Schmidt's conjecture is given by
$\overline{\psi}_{n,T}<0<\underline{\psi}_{n,T+2}$. 
In view of \eqref{eq:aequi} that is in turn equivalent to
$\lambda_{n,T+2}<1/n<\widehat{\lambda}_{n,T}$. 
In this context recall that for the regular graph we have $\widehat{\lambda}_{n,T}=\lambda_{n,T+1}$.
We will 
investigate below how the quantities $\lambda_{n,j}$ for
the regular graph in fixed dimension $n$ in depend on the parameter $\lambda\geq 1/n$. 
Concretely when we ask for the largest index $j$ 
such that $\lambda_{n,j}$ is larger than $1/n$ in such intervals, the above correspondence
indicates the close connection to Schmidt's conjecture. 
Indeed Theorem~\ref{moshthm} will provide the link.
We start with an easy but important preparatory observation.

\begin{proposition} \label{lagegrund}
Let $n\geq 2$ and $1\leq j\leq n+2$. Then the quantities
$\lambda_{n,j}(\lambda)=\widehat{\lambda}_{n,j-1}(\lambda)$ for the regular graph in dimension $n$
with parameter $\lambda$ satisfies
\[
\frac{\lambda}{(1+\lambda)^{j-1}}\leq \lambda_{n,j}(\lambda)\leq \lambda^{2-j}, \qquad \lambda\in{[1/n,\infty]}.
\]
In particular if $j\geq 3$ then $\lambda_{n,j}(\lambda)$ tends to $0$ as $\lambda$ tends to infinity.
\end{proposition}

\begin{proof}
The left inequality was already established in Corollary~\ref{korola}. For the right estimate
observe $\widehat{\lambda}_{n}(\lambda)=\lambda_{n,2}(\lambda)\leq 1$ always holds by 
\eqref{eq:1}.
Together with the constant quotients property \eqref{eq:quot} we have 
$\lambda_{n,j}(\lambda)=\widehat{\lambda}_{n}(\widehat{\lambda}_{n}/\lambda)^{j-2}\leq \lambda^{2-j}$, which
clearly tends to $0$ for $j\geq 3$ as $\lambda\to\infty$.
\end{proof}

In particular $\lambda_{n,j}(\lambda)\thicksim \lambda^{2-j}$ for $1\leq j\leq n+2$ as $\lambda\to\infty$. Dually,
if we denote by $w_{n,j}(w)$ the constants $w_{n,j}$ for the regular graph for
the parameter $w_{n,1}=w$,
then with \eqref{eq:platz} we deduce $w_{n,j}(w)\thicksim w^{(n-j+1)/n}$ as $w\to\infty$
for $1\leq j\leq n+2$. The next theorem provides more detailed information 
on the functions $\lambda_{n,j}(\lambda)$ in \eqref{eq:zuordnung}.

\begin{theorem} \label{bergzwerg}
Let $j\geq 3$ and $n\geq j-2$ be integers. If $n\geq 2j-2$, then there exist 
$\tilde{\lambda}\in{(1/n,n)}$ with the following properties.
The regular graph in dimension $n$ with parameter $\lambda$
satisfies $\lambda_{n,j}(\lambda)>1/n$
for $\lambda\in(1/n,\tilde{\lambda})$, $\lambda_{n,j}(\lambda)=1/n$ for $\lambda\in\{1/n,\tilde{\lambda}\}$
and $\lambda_{n,j}(\lambda)<1/n$ for $\lambda\in(\tilde{\lambda},\infty]$.
If on the other hand $n\leq 2j-3$, then for all $\lambda\in{(1/n,\infty]}$ the regular graph
in dimension $n$ with parameter $\lambda$ satisfies $\lambda_{n,j}(\lambda)<1/n$.
\end{theorem}

It is easy to check the following consequence of Theorem~\ref{bergzwerg}.

\begin{corollary} \label{coronalight}
Precisely in case of $n\leq 3$ 
none of the functions $\lambda_{n,j}(\lambda)-1/n$
changes sign on $\lambda\in{(1/n,\infty)}$.
\end{corollary}

The claims of Theorem~\ref{bergzwerg} and Corollary~\ref{coronalight} are 
(to some degree) visible 
in Figure~3 for $n=8$.


\begin{figure}
\centering
 \includegraphics[width=0.5\textwidth,natwidth=650,natheight=650]{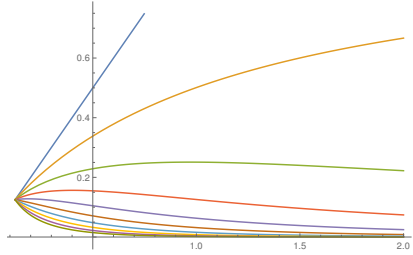}
 \caption{The functions $\lambda_{8,1}(\lambda),\ldots,
\lambda_{8,10}(\lambda)$ in the interval $\lambda\in[1/8,2]$}
\end{figure}

\begin{remark} \label{zwergberg}
For $j\in\{1,2\}$ and $n\geq 2$ clearly we have $\lambda_{n,j}(\lambda)>1/n$ for all $\lambda\in{(1/n,\infty]}$
by \eqref{eq:1} and \eqref{eq:6}, with equality in both inequalities 
only for $\lambda=1/n$. See also Proposition~\ref{propodil}. 
A similar dual argument shows $\lambda_{n,j}(\lambda)<1/n$ for $j\in\{n+1,n+2\}$, as we will carry out in the proof. 
In particular for $n=2$ it is clear that $\lambda_{2,1}(\lambda)>\lambda_{2,2}(\lambda)>1/2>
\lambda_{n,3}(\lambda)>\lambda_{n,4}(\lambda)$ for all $\lambda>1/2$, and it can be shown
easily that all functions $\lambda_{2,i}(\lambda)$ are monotonic on $[1/n,\infty]$, see also Figure~2. On the other hand, for $n=3$ the above argument is already 
too weak to imply $\lambda_{3,3}(\lambda)<1/3$ for all $\lambda>1/3$, as 
Theorem~\ref{bergzwerg} does. 
\end{remark}

Moreover it should be true that
the derivative of $\lambda_{n,j}(\lambda)$ with respect to the parameter $\lambda$ changes sign at most once, and precisely for $3\leq j< \frac{n+3}{2}$, somewhere 
in the interval $(1/n,\tilde{\lambda})$ with $\tilde{\lambda}$ from Theorem~\ref{bergzwerg}. 
However, we omit a most likely cumbersome proof. 
From Theorem~\ref{bergzwerg} it is not hard to deduce explicit examples for Schmidt's property 
if $T$ does not exceed roughly $n/2$.

\begin{theorem} \label{moshthm}
Let $n\geq 2$ be an integer. Then for any $1\leq T\leq \lfloor n/2\rfloor$ there 
exists a non-empty subinterval $I=I(T)$ of $(1/n,n)$ such that for all $\lambda\in{I}$ the 
regular graph in dimension $n$ with parameter $\lambda$ satisfies 
\[
\widehat{\lambda}_{n,T}(\lambda)>\frac{1}{n}, \qquad \lambda_{n,T+2}(\lambda)<\frac{1}{n}.
\]
In other words for any pair $(n,T)$ with $1\leq T\leq \lfloor n/2\rfloor$ there exist
$\underline{\zeta}$ that induce the regular graph and satisfy Schmidt's property for $(n,T)$.
For $T>\lfloor n/2\rfloor$ such $\underline{\zeta}$ does not exist.
\end{theorem}

\begin{proof}
First let $3\leq j\leq \lfloor n/2\rfloor+1$. 
Then the first case of Theorem~\ref{bergzwerg} applies
and yields $\lambda_{n,j}(\tilde{\lambda})=1/n$ and $\lambda_{n,j}(t)>1/n$ for
some $\tilde{\lambda}>1/n$ and $t\in{(1/n,\tilde{\lambda})}$. 
Since $\lambda_{n,j+1}<\lambda_{n,j}$ unless both are equal to $\lambda=1/n$, we
have $\lambda_{n,j+1}(\tilde{\lambda})<1/n$. Hence by continuity of the function 
$\lambda_{n,j+1}(\lambda)$ in the parameter $\lambda$, there exists 
some non-empty interval $J=J(j)=(\delta-\epsilon,\delta)$ such that for 
$t_{0}\in{J}$ the inequalities $\lambda_{n,j+1}(t_{0})<1/n<\lambda_{n,j}(t_{0})$ are satisfied. 
Since in the regular graph $\widehat{\lambda}_{n,j-1}=\lambda_{n,j}$ holds by \eqref{eq:galeich}, 
the claim follows for $T\geq 2$ with $T=j-1$, and the fact that
$\underline{\zeta}$ inducing the corresponding regular graphs exist as mentioned above.
For $T=1$, a very similar argument applies with $j=2$. We may take any value $\lambda$ 
sufficiently large that $\lambda_{n,3}(\lambda)-1/n<0$,
observing $\lambda_{n,2}(\lambda)>1/n$ for $\lambda>1/n$
but $\lambda_{n,3}(\lambda)-1/n$ changes sign somewhere in $(1/n,n)$. 
Finally, concerning the claim for $T>\lfloor n/2\rfloor$, 
suitable $\underline{\zeta}$ cannot exist
since $\lambda_{n,T+1}(\lambda)=\widehat{\lambda}_{n,T}(\lambda)<1/n$ 
for all $\lambda>1/n$ by the last claim of Theorem~\ref{bergzwerg}.
\end{proof}

\begin{remark}
For $T<n/e$ the claim concerning Schmidt's property
could be derived directly from Proposition~\ref{lagegrund} instead 
of the deeper Theorem~\ref{bergzwerg},
where $e$ is Euler's number. 
\end{remark}

\section{Implications of Conjecture~\ref{schmsum} for uniform approximation} \label{implic}

In this section we restrict to the case of 
successive powers $(\zeta,\zeta^{2},\ldots,\zeta^{n})$.
We will write $w_{n,j}(\zeta)$ for $w_{n,j}(\zeta,\zeta^{2},\ldots,\zeta^{n})$ 
and similarly for $\widehat{w}_{n,j}(\zeta),\lambda_{n,j}(\zeta),\widehat{\lambda}_{n,j}(\zeta)$.
We will also consider related constants connected
to approximation by algebraic numbers. For a given
real number $\zeta$, let $w_{n}^{\ast}(\zeta)$ be
the supremum of $\nu$ such that
\[
0<\vert \zeta-\alpha\vert \leq H(\alpha)^{-\nu-1}
\]
has infinitely many
real algebraic solutions $\alpha$ of degree at most $n$. Here $H(\alpha)$ is the 
height of the irreducible minimal polynomial $P$ of $\alpha$ over $\mathbb{Z}[X]$,
which is the maximum modulus among its coefficients. 
Similarly let the uniform constant $\widehat{w}_{n}^{\ast}(\zeta)$ 
be the supremum of real $\nu$ for which the system
\[
H(\alpha)\leq X, \qquad 0<\vert \zeta-\alpha\vert \leq H(\alpha)^{-1}X^{-\nu}
\]
has a solution as above for all large values of $X$.
For all $n\geq 1$ and all real $\zeta$, the estimates 
\begin{equation} \label{eq:fischers}
w_{n}^{\ast}(\zeta)\leq w_{n}(\zeta)\leq w_{n}^{\ast}(\zeta)+n-1, \qquad 
\widehat{w}_{n}^{\ast}(\zeta)\leq \widehat{w}_{n}(\zeta)\leq \widehat{w}_{n}^{\ast}(\zeta)+n-1,
\end{equation}
are well-known, see \cite[Lemma~A8]{bugbuch}.
We aim to establish a conditional improvement of the known upper bound for the exponents
$\widehat{w}_{n}(\zeta), \widehat{w}_{n}^{\ast}(\zeta)$ valid for all transcendental real $\zeta$, under the assumption of Conjecture~\ref{schmsum}. The bound
$\widehat{w}_{n}(\zeta)\leq 2n-1$ was given by Davenport and Schmidt~\cite{davsh}.
This has recently been refined in~\cite[Theorem 2.1]{buschl} to
\begin{equation} \label{eq:tomcat3}
\widehat{w}_{n}(\zeta)\leq  
 n - \frac{1}{2} + \sqrt{n^{2}-2n+ \frac{5}{4}}.
\end{equation}
For large $n$ the right hand side in \eqref{eq:tomcat3}
is of order $2n-3/2+o(1)$.
For $n=3$, the stronger estimate
\begin{equation} \label{eq:tomcat4}
\widehat{w}_{3}(\zeta)\leq 3+\sqrt{2}\approx 4.4142...
\end{equation}
was established in~\cite[Theorem 2.1]{buschl}.
For $n=2$, the bound in~\eqref{eq:tomcat3} is best possible 
as proved by Roy, see \cite{royroy}.
Our main result of this section is the following asymptotic estimation,
conditioned on Conjecture~\ref{schmsum}.

\begin{theorem} \label{reggraph}
	Suppose Conjecture~\ref{schmsum} holds for every $n\geq 2$. 
	Let $\tau\approx 0.5693$ be the solution $y\in{(0,1)}$ of $ye^{1/y}=2\sqrt{e}$, where $e$ is Euler's number,
	and put $\Delta:=\log(2/\tau)+1\approx 2.2564$. Then for any $\epsilon>0$ there
	exists $n_{0}=n_{0}(\epsilon)$ such that for all real transcendental numbers $\zeta$ we have
	\begin{equation} \label{eq:ohne}
	\widehat{w}_{n}^{\ast}(\zeta)\leq 2n-\Delta+\epsilon, \qquad n\geq n_{0}.
	\end{equation}
	The same bound holds for $\widehat{w}_{n}(\zeta)$ unless $w_{n-2}(\zeta)<w_{n-1}(\zeta)=w_{n}(\zeta)$.
	In any case we have
	\begin{equation} \label{eq:mit}
	\widehat{w}_{n}(\zeta)\leq 2n-2, \qquad n\geq 10.
	\end{equation}
\end{theorem}

Furthermore, in Section~\ref{4proofs} we will derive conditioned concrete 
upper bounds for $\widehat{w}_{n}(\zeta), \widehat{w}_{n}^{\ast}(\zeta)$ 
for certain values of $n$, see \eqref{eq:kanada} below. 
We close this section with another related result, whose proof
will be omitted as it is very similar to that of Theorem~\ref{reggraph}.
Assume that the estimate
\begin{equation} \label{eq:glueck}
\widehat{w}_{n}\leq n^{\frac{1}{n+1}}w_{n}^{\frac{n}{n+1}}
\end{equation}
is satisfied. 
Then for every $\epsilon>0$ there exists $n_{0}=n_{0}(\epsilon)$ such that
\begin{equation} \label{eq:wegweg}
\widehat{w}_{n}(\zeta)\leq 2n-1-\log 2+\epsilon, \qquad n\geq n_{0}.
\end{equation}
Observe that \eqref{eq:wegweg} is still stronger than \eqref{eq:tomcat3}, although it is
weaker than \eqref{eq:ohne}. On the other hand,
we will see in Section~\ref{graph} that the involved assumption \eqref{eq:glueck} is 
reasonably weaker than the assumption of Conjecture~\ref{schmsum} in Theorem~\ref{reggraph}.

\section{Conditioned results under assumption of another conjecture} \label{andere}

\subsection{Uniform approximation}
Let $n\geq 1$ an integer and $\zeta$ a real number.
We call $P\in{\mathbb{Z}[X]}$ of degree at most $n$
a best approximation for $(n,\zeta)$ if there is no $Q\in{\mathbb{Z}[X]}$
of degree at most $n$ with strictly smaller height
 $H(Q)<H(P)$ that satisfies $\vert Q(\zeta)\vert<\vert P(\zeta)\vert$. Obviously
every real transcendental $\zeta$ induces 
a sequence of best approximation polynomials $P_{1},P_{2},\ldots$ with
$\vert P_{1}(\zeta)\vert>\vert P_{2}(\zeta)\vert>\cdots$ and $H(P_{1})\leq H(P_{2})\leq \cdots$.
Similarly for $\underline{\zeta}=(\zeta_{1},\ldots,\zeta_{n})$ define best
approximations for $(n,\underline{\zeta})$ for the linear forms in $\underline{\zeta}$.

\begin{conjecture} \label{bestapp}
For any $n\geq 1$ and any real transcendental $\zeta$, there exist infinitely many $k$ such that 
$n+1$ successive best approximations $P_{k},P_{k+1},\ldots,P_{k+n}$ for $(n,\zeta)$ 
are linearly independent (i.e. the coefficient vectors span the entire space $\mathbb{R}^{n+1}$).
\end{conjecture}

\begin{remark} \label{markenrehe}
The claim is known to hold for $n=2$. More generally, for any $n$ there 
are three linearly independent 
consecutive best approximations infinitely often, see~\cite{ss}. On the
other hand, Moshchevitin~\cite{mosh}
proved the existence of counterexamples for the analogous claim for
vectors $\underline{\zeta}\in{\mathbb{R}^{n}}$ that are
$\mathbb{Q}$-linearly independent together with $\{1\}$, for $n>2$. Vectors can even be chosen
such that the $(n+1)\times (n+1)$-matrix whose columns are
formed by $n+1$ successive best approximation
vectors has rank at most $3$ for all large $k$.
However, it seems plausible that such vectors cannot lie on the Veronese curve.
\end{remark}

\begin{theorem} \label{zwischen}
For any $n\geq 2$ and any real vector $\underline{\zeta}$ linearly independent over $\mathbb{Q}$
together with $\{1\}$, we have
\begin{equation} \label{eq:fastfertig}
w_{n,3}(\underline{\zeta})\geq \frac{\widehat{w}_{n}(\underline{\zeta})^{2}}{w_{n}(\underline{\zeta})}.
\end{equation}
If $(n,\underline{\zeta})$ satisfies the assumption of Conjecture~\ref{bestapp} then 
\begin{equation} \label{eq:heute} 
w_{n,i}(\underline{\zeta})\geq 
\frac{\widehat{w}_{n}(\underline{\zeta})^{i-1}}{w_{n}(\underline{\zeta})^{i-2}}, \qquad 1\leq i\leq n+1,
\end{equation}
and
\begin{equation} \label{eq:99lb}
w_{n}(\underline{\zeta})\geq 
\widehat{w}_{n}(\underline{\zeta})\left(\frac{\widehat{w}_{n}(\underline{\zeta})-1}{n-1}\right)^{\frac{1}{n-1}}.
\end{equation}
Analogous claims of \eqref{eq:fastfertig} and \eqref{eq:heute}
hold for $\lambda_{n,j},\widehat{\lambda}_{n,j}$ with respect to the obvious 
dual definition of the best approximations and Conjecture~\ref{bestapp}, and \eqref{eq:99lb}
has to be replaced by 
\begin{equation} \label{eq:feierabend}
\lambda_{n}(\underline{\zeta})\geq 
\widehat{\lambda}_{n}(\underline{\zeta})\cdot 
\left(\frac{(n-1)\widehat{\lambda}_{n}(\underline{\zeta})}{1-\widehat{\lambda}_{n}(\underline{\zeta})}\right)^{\frac{1}{n-1}}.
\end{equation}
\end{theorem}

For $n=2$, the estimate \eqref{eq:99lb} is unconditioned by Remark~\ref{markenrehe} and yields 
the inequality $w_{2}(\underline{\zeta})\geq \widehat{w}_{2}(\underline{\zeta})(\widehat{w}_{2}(\underline{\zeta})-1)$
known by Laurent~\cite{taurent}. There is equality in all inequalities of Theorem~\ref{zwischen}
for $(\zeta,\zeta^{2})$ when $\zeta$ is an extremal number defined by Roy,
see for example~\cite{royroy}. 
See also Moshchevitin~\cite[Section~3]{mosh3}
for  results related to \eqref{eq:99lb} and \eqref{eq:feierabend}.
For us the main purpose of Theorem~\ref{zwischen} is the connection
to uniform approximation, portrayed in the following theorem.  

\begin{theorem} \label{anderer}
Assume Conjecture~\ref{bestapp} is true. 
Then \eqref{eq:wegweg} holds.
\end{theorem}

\section{Proofs}

\subsection{Preliminary results}  \label{graph}

In this section
we establish several identities involving 
the exponents $\lambda_{n,j},\widehat{\lambda}_{n,j},w_{n,j},\widehat{w}_{n,j}$
in the regular graph, 
to prepare the proofs of the main results. 
They are essentially derived by algebraic rearrangements of
the identity
\begin{equation} \label{eq:total}
\frac{(\lambda_{n}+1)^{n+1}}{\lambda_{n}}=
\frac{(\widehat{\lambda}_{n,n+1}+1)^{n+1}}{\widehat{\lambda}_{n,n+1}},
\end{equation}
which was proved in~\cite[(95) in~Section~3]{j1}. 
In view of \eqref{eq:total} we define the auxiliary functions
\begin{equation} \label{eq:effen}
f_{n}(x):=\frac{(1+x)^{n+1}}{x}.
\end{equation}
It is easily verified that $f_{n}$ decays on $(0,1/n)$ and increases on $(1/n,\infty)$.
Hence we see that for given $\lambda_{n}\in{[1/n,\infty]}$, the constant
$\widehat{\lambda}_{n,n+1}$ is the unique solution of \eqref{eq:total} in the interval $[0,1/n]$.
Observe that by \eqref{eq:total} and the constant quotients \eqref{eq:quot}, the constants
$\lambda_{n}=\lambda$ and $\lambda_{n,j}(\lambda)$ satisfy the implicit equation
\begin{equation} \label{eq:impgl}
\frac{(1+\lambda)^{n+1}}{\lambda}=
\frac{\left(1+\lambda^{1-\frac{n+1}{j-1}}\lambda_{n,j}(\lambda)^{\frac{n+1}{j-1}}\right)^{n+1}}
{\lambda^{1-\frac{n+1}{j-1}}\lambda_{n,j}(\lambda)^{\frac{n+1}{j-1}}}.
\end{equation}
Moreover from \eqref{eq:total} and \eqref{eq:quot} we infer 
\begin{equation} \label{eq:sim}
\widehat{\lambda}_{n}=\lambda_{n}^{\frac{n}{n+1}}\widehat{\lambda}_{n,n+1}^{\frac{1}{n+1}}
=\lambda_{n}\left(\frac{\widehat{\lambda}_{n,n+1}}{\lambda_{n}}\right)^{\frac{1}{n+1}}.
\end{equation}
By combining \eqref{eq:total} with \eqref{eq:sim}, after 
some rearrangements we derive an implicit polynomial equation 
involving $\lambda_{n}$ and $\widehat{\lambda}_{n}$ of the form
\begin{equation} \label{eq:notnice}
(\widehat{\lambda}_{n}-1)\lambda_{n}^{n}+ 
\widehat{\lambda}_{n}\lambda_{n}^{n-1}-\widehat{\lambda}_{n}^{n+1}=0,
\end{equation}
where in the special case $\lambda_{n}=\infty$ we have to 
put $\widehat{\lambda}_{n}=1$.
Noticing that $\widehat{\lambda}_{n}=\lambda_{n}$ is a solution
of \eqref{eq:notnice} not of interest, we can decrease the degree by one
\[
\lambda_{n}^{n-1}(\widehat{\lambda}_{n}-1)+\widehat{\lambda}_{n}^{2}
\frac{\lambda_{n}^{n-1}-\widehat{\lambda}_{n}^{n-1}}{\lambda_{n}-\widehat{\lambda}_{n}}=\lambda_{n}^{n-1}(\widehat{\lambda}_{n}-1)+\widehat{\lambda}_{n}^{2}
(\lambda_{n}^{n-2}+\lambda_{n}^{n-3}\widehat{\lambda}_{n}+\cdots+
\widehat{\lambda}_{n}^{n-2})
=0.
\]
Now we want to establish the dual results.
One can either proceed similarly as in~\cite{j1} for \eqref{eq:total},
or immediately apply \eqref{eq:platz} 
to \eqref{eq:total}, to derive
\begin{equation} \label{eq:j1gl1}
\frac{(1+w_{n})^{n+1}}{w_{n}^{n}}=
\frac{(1+\widehat{w}_{n,n+1})^{n+1}}{\widehat{w}_{n,n+1}^{n}}
\end{equation}
for the regular graph. Observe that $\widehat{w}_{n,n+1}=1/\lambda_{n}\in{[0,n]}$ by
\eqref{eq:platz} and \eqref{eq:dir1}, whereas $w_{n}\in{[n,\infty]}$ by \eqref{eq:dir2}.
In particular it is not hard to see that for given $w_{n}\in{[n,\infty]}$ the approximation constant
$\widehat{w}_{n,n+1}$ is the unique real solution of \eqref{eq:j1gl1} in the interval $[0,n]$.
Moreover again for the regular graph all quotients
$w_{n,j}/w_{n,j+1}=w_{n,j}/\widehat{w}_{n,j}$
coincide for $1\leq j\leq n+1$, where we put $w_{n,n+2}:=\widehat{w}_{n,n+1}$. 
This yields
\begin{equation} \label{eq:j1gl2}
\widehat{w}_{n}=w_{n}^{\frac{n}{n+1}}\widehat{w}_{n,n+1}^{\frac{1}{n+1}}
=w_{n}\left(\frac{\widehat{w}_{n,n+1}}{w_{n}}\right)^{\frac{1}{n+1}}.
\end{equation}
From \eqref{eq:j1gl2} and the most right inequality of \eqref{eq:16} we obtain
\eqref{eq:glueck},
where equality holds only in case of $\widehat{w}_{n,n+1}=n$ or equivalently $w_{n}=n$.
Expressing $\widehat{w}_{n}$ in terms of $w_{n},\widehat{w}_{n}$ by rearranging
\eqref{eq:j1gl2} and inserting in \eqref{eq:j1gl1}, some further 
rearrangements lead to the nice implicit equation
\begin{equation} \label{eq:pasta}
w_{n}-\widehat{w}_{n}+1=\left(\frac{w_{n}}{\widehat{w}_{n}}\right)^{n}.
\end{equation}
We summarize the above observations in a proposition.

\begin{proposition}  \label{psiprop}
	The function $\phi_{n}$ coincides with the unique solution of $\widehat{w}_{n}$ in \eqref{eq:pasta} 
	in terms of $w_{n}$ in the interval $[n,w_{n})$, unless $w_{n}=\phi_{n}(w_{n})=\widehat{w}_{n}=n$. 
	The function $\vartheta_{n}$ coincides with the 
	unique solution of $\widehat{\lambda}_{n}$ in \eqref{eq:notnice} in terms of $\lambda_{n}$ in 
	the interval $[1/n,\lambda_{n})$, unless $\lambda_{n}=\vartheta_{n}(\lambda_{n})=\widehat{\lambda}_{n}=1/n$.
\end{proposition}

\begin{proof}
	The asserted uniqueness can be easily proved. It has been established 
	that \eqref{eq:pasta} and \eqref{eq:notnice} are satisfied and the claim on the intervals follows
	from \eqref{eq:dir1} and \eqref{eq:dir2}.
\end{proof}

We remark that similarly to \eqref{eq:impgl},
one can obtain an implicit equation involving $w_{n}$ and 
$w_{n,j}=\widehat{w}_{n,j-1}$ for $2\leq j\leq n+2$, and
dual interpretations of 
Theorem~\ref{bergzwerg} and Remark~\ref{zwergberg} provide some information on the
monotonicity of the functions $w_{n,j}$ in dependence of $w_{n}$. We do not carry this out.

\subsection{Proofs of Section~\ref{newregular}} \label{proof1}

For the first proof recall the functions $f_{n}$ from \eqref{eq:effen} and their properties.

\begin{proof}[Proof of Proposition~\ref{propodil}]
By the assumptions the regular graphs with parameter $\lambda$ in dimension $n_{1},n_{2}$ 
are well-defined (and exist due to Roy~\cite{roy5}).
Since in the regular graph the quotients \eqref{eq:quot} coincide, it suffices to 
prove that $\lambda_{n,2}(\lambda)=\widehat{\lambda}_{n}(\lambda)$
decreases for fixed $\lambda$ as $n$ increases. 

Recall the functions $f_{n}$ defined in Section~\ref{newregular}.
We have $f_{n+1}(\lambda)/f_{n}(\lambda)=1+\lambda$ and hence in view of \eqref{eq:total} also
\begin{equation} \label{eq:wild}
\frac{f_{n+1}(\widehat{\lambda}_{n+1,n+2}(\lambda))}{f_{n}(\widehat{\lambda}_{n,n+1}(\lambda))}=1+\lambda.
\end{equation}
On the other hand we claim that 
\begin{equation} \label{eq:wechsel}
\widehat{\lambda}_{n+1,n+2}(\lambda)< \widehat{\lambda}_{n,n+1}(\lambda).
\end{equation}
In case of $\widehat{\lambda}_{n,n+1}(\lambda)>1/(n+1)$ this is trivial since 
$\widehat{\lambda}_{n+1,n+2}(\lambda)\leq 1/(n+1)$. If otherwise $\widehat{\lambda}_{n,n+1}(\lambda)\leq 1/(n+1)$,
then \eqref{eq:wechsel} follows from the decay of the function $f_{n+1}$ on $(0,1/(n+1))$ and
$f_{n+1}(\widehat{\lambda}_{n,n+1}(\lambda))/f_{n}(\widehat{\lambda}_{n,n+1}(\lambda))
=1+\widehat{\lambda}_{n,n+1}(\lambda)\leq 1+\lambda$, in combination with \eqref{eq:wild}.
From \eqref{eq:wechsel} we deduce
\[
(1+\widehat{\lambda}_{n+1,n+2}(\lambda))^{n+2}\leq (1+\widehat{\lambda}_{n,n+1}(\lambda))^{n+2}
= (1+\widehat{\lambda}_{n,n+1}(\lambda))^{n+1}\frac{f_{n+1}(\widehat{\lambda}_{n,n+1}(\lambda))}
{f_{n}(\widehat{\lambda}_{n,n+1}(\lambda))}.
\]
Observe the left and middle quantities are the nominators 
of $f_{n+1}(\widehat{\lambda}_{n,n+1}(\lambda))$ and 
$f_{n+1}(\widehat{\lambda}_{n+1,n+2}(\lambda))$, respectively.
Together with \eqref{eq:wild}, for the according denominators we infer
\begin{equation} \label{eq:esel}
\frac{\widehat{\lambda}_{n,n+1}(\lambda)}{\widehat{\lambda}_{n+1,n+2}(\lambda)}
>\frac{1+\lambda}{\widehat{\lambda}_{n,n+1}(\lambda)}.
\end{equation}
The identities \eqref{eq:sim} for $n,n+1$ yield
\begin{align*}
\widehat{\lambda}_{n}(\lambda)&= \lambda^{n/(n+1)}\widehat{\lambda}_{n,n+1}(\lambda)^{1/(n+1)},  \\
\widehat{\lambda}_{n+1}(\lambda)&= \lambda^{(n+1)/(n+2)}\widehat{\lambda}_{n+1,n+2}(\lambda)^{1/(n+2)}.
\end{align*} 
Taking quotients, with $(n+1)/(n+2)-n/(n+1)=1/(n+1)-1/(n+2)=(n+1)^{-1}(n+2)^{-1}$ we get
\[
\frac{\widehat{\lambda}_{n}(\lambda)}{\widehat{\lambda}_{n+1}(\lambda)}
\geq \lambda^{-\frac{1}{(n+1)(n+2)}}\widehat{\lambda}_{n,n+1}(\lambda)^{\frac{1}{(n+1)(n+2)}}
\left(\frac{\widehat{\lambda}_{n,n+1}(\lambda)}{\widehat{\lambda}_{n+1,n+2}(\lambda)}\right)^{\frac{1}{n+2}}.
\]
Inserting the bound \eqref{eq:esel}, for the last expression we obtain
\begin{equation} \label{eq:rhswichtig}
\frac{\widehat{\lambda}_{n}(\lambda)}{\widehat{\lambda}_{n+1}(\lambda)}
\geq \lambda^{-\frac{1}{(n+1)(n+2)}}\widehat{\lambda}_{n,n+1}(\lambda)^{\frac{1}{(n+1)(n+2)}}
\left(\frac{1+\lambda}{\widehat{\lambda}_{n,n+1}(\lambda)}\right)^{\frac{1}{n+2}}.
\end{equation}
One readily checks that the right hand side in \eqref{eq:rhswichtig} equals $1$, since this 
is equivalent to $f_{k}(\lambda)=f_{k}(\widehat{\lambda}_{n,n+1}(\lambda))$, which is \eqref{eq:total}.
This finishes the proof.
\end{proof}

\begin{proof}[Proof of Corollary~\ref{korola}]
It was shown in~\cite[Proposition~5]{ss} that
we have $\widehat{\lambda}_{n}(\lambda)/\lambda> (\lambda+1)^{-1}$
in the regular graph with parameter $\lambda_{n}=\lambda$. On the other hand the quotients 
$\lambda_{n,j}/\lambda_{n,j+1}$ are identical for all $1\leq j\leq n+1$ by \eqref{eq:quot}. Hence
\[
\widehat{\lambda}_{n,j-1}(\lambda)= \lambda_{n,j}(\lambda)=
\lambda\left(\frac{\widehat{\lambda}_{n}(\lambda)}{\lambda}\right)^{j-1}\geq \frac{\lambda}{(1+\lambda)^{j-1}}.
\]
In Proposition~\ref{propodil} we proved that the values $\widehat{\lambda}_{n,j-1}(\lambda)= \lambda_{n,j}(\lambda)$
decay as $n$ increases, hence the limit of $\lambda_{n,j}(\lambda)$ as $n\to\infty$
exists and equals at least the given quantity. We have to show
equality. Again as all the quotients 
$\lambda_{n,j-1}/\lambda_{n,j}$ are identical, 
it obviously suffices to show this for $j=2$.
For $\lambda,\widehat{\lambda}_{n}(\lambda)$ as above define $\alpha(n)$ implicitly by 
\begin{equation} \label{eq:alpha}
\widehat{\lambda}_{n}(\lambda)=\alpha(n)\frac{\lambda}{1+\lambda}.
\end{equation}
Then the sequence $\alpha(n)\geq 1$ decreases to some limit at least $1$
and we have to show $\lim_{n\to\infty} \alpha(n)=1$. Observe a rearrangement of \eqref{eq:sim}
and \eqref{eq:alpha} yield
\[
\widehat{\lambda}_{n,n+1}(\lambda)=\lambda\left(\frac{\widehat{\lambda}_{n}(\lambda)}{\lambda}\right)^{n+1}
=\lambda\left(\frac{\alpha(n)}{1+\lambda}\right)^{n+1}.
\]
Inserting the right hand side in the identity \eqref{eq:total}, 
elementary rearrangements lead to
\begin{equation} \label{eq:toni}
\alpha(n)= 1+\lambda\left(\frac{\alpha(n)}{1+\lambda}\right)^{n+1}.
\end{equation}
If we had $\lim_{n\to\infty} \alpha(n)\geq \lambda+1$ then $\widehat{\lambda}_{n}(\lambda)\geq \lambda
=\lambda_{n}(\lambda)$,
contradiction. Thus $\lim_{n\to\infty} \alpha(n)<\lambda+1$. Hence the right hand side of \eqref{eq:toni} converges to $1$ as $n\to\infty$, 
and thus the left hand side does as well. This completes the proof. 
\end{proof}

For the proof of Theorem~\ref{bergzwerg} we consider $\lambda$ in small 
intervals of the form $(1/n,1/n+\epsilon)$. 

\begin{proof}[Proof of Theorem~\ref{bergzwerg}]
Clearly $\lambda_{n,j}=1/n$ for all $1\leq j\leq n+2$ if $\lambda=1/n$.
Further observe that by the most left inequalities of
\eqref{eq:11} and \eqref{eq:16}, and \eqref{eq:platz},
we have $\lambda_{n,n+1}(\lambda)=\widehat{w}_{n}(\lambda)^{-1}\leq 1/n$
and $\lambda_{n,n+2}=w_{n}(\lambda)^{-1}\leq 1/n$ Equality holds
only if the quantities equal $1/n$ anyway, where we put $w_{n}(\lambda)$ for
the value $w_{n}$ induced for the regular graph with parameter $\lambda_{n,1}=\lambda$. 
Thus we can restrict to $2\leq j\leq n$. 

So let $n\geq 1$ and $2\leq j\leq n$ be arbitrary but fixed. Write $\lambda_{n}=\lambda=\alpha/n$ 
for $\alpha>1$, where we consider only $\alpha$ slightly larger than $1$.
Then \eqref{eq:impgl} becomes
\begin{equation} \label{eq:daoben}
\frac{\left(1+\frac{\alpha}{n}\right)^{n+1}}{\frac{\alpha}{n}}
= \frac{\left(1+\left(\frac{\alpha}{n}\right)^{1-\frac{n+1}{j-1}}\lambda_{n,j}(\frac{\alpha}{n})^{\frac{n+1}{j-1}} \right)^{n+1}}
{\left(\frac{\alpha}{n}\right)^{1-\frac{n+1}{j-1}}\lambda_{n,j}(\frac{\alpha}{n})^{\frac{n+1}{j-1}}}.
\end{equation}
We ask for which values of $j$ it is possible to have $\lambda_{n,j}(\frac{\alpha}{n})=1/n$ for some $\alpha>1$.
So we insert $\lambda_{n,j}(\frac{\alpha}{n})=1/n$ in \eqref{eq:daoben}, 
and rearrange \ref{eq:daoben} in the following way. We multiply
with $\alpha/n$, then divide by the nominator
of the right hand side and take the $(n+1)$-st root. After further elementary rearrangements
and simplification, we end up with the equivalent identity
\begin{equation} \label{eq:furien}
n=\frac{\alpha-\alpha^{\frac{j-n-1}{j-1}}}{\alpha^{\frac{1}{j-1}}-1}.
\end{equation}
Let $\theta:=\alpha^{\frac{1}{j-1}}$. Clearly $\theta>1$ is equivalent to $\alpha>1$.
Furthermore \eqref{eq:furien} is equivalent to
\begin{equation} \label{eq:honign}
n=\theta^{j-1-n}\frac{\theta^{n}-1}{\theta-1}=\theta^{j-2}+\theta^{j-3}+\cdots+\theta^{j-1-n}=:\chi_{n,j}(\theta).
\end{equation}
By construction $\chi_{n,j}(1)=n$. First consider $n\leq 2j-3$ or equivalently $j\geq \frac{n+3}{2}$. 
We calculate
\[
\chi_{n,j}^{\prime}(t)=(j-2)t^{j-3}+(j-3)t^{j-4}+\cdots+(j-n-1)t^{j-n-2},
\]
and
\[
\chi_{n,j}^{\prime\prime}(t)=(j-2)(j-3)t^{j-4}+(j-3)(j-4)t^{j-5}+\cdots+(j-1-n)(j-2-n)t^{j-n-3}.
\]
It is easy to verify $\chi_{n,j}^{\prime\prime}(t)>0$ for all $t>0$. Indeed any expression in the sum is non-negative, and for $j\geq 4$ the first and for $j=3$ the last is strictly positive. Hence it suffices to show
$\chi_{n,j}^{\prime}(1)>0$ to see that $\chi_{n,j}(t)>n$ for all $t>1$. 
Indeed, for $j\geq \frac{n+3}{2}$ we verify
\begin{equation} \label{eq:kardinal}
\chi_{n,j}^{\prime}(1)=(j-2)+(j-3)+\cdots+(j-n-1)=nj-\sum_{i=2}^{n+1} i= nj-\frac{n^{2}+3n}{2}\geq 0.
\end{equation}
We conclude $\lambda_{n,j}(\lambda)\neq 1/n$ for all $\lambda>1/n$. By the continuity of $\lambda_{n,j}$, we must have either $\lambda_{n,j}(\lambda)<1/n$ for 
all $\lambda>1/n$ or $\lambda_{n,j}(\lambda)>1/n$ for all $\lambda>1/n$.
However, since $j\geq 3$, we can exclude the latter since in Proposition~\ref{lagegrund} we showed
\begin{equation} \label{eq:limmes}
\lim_{\lambda\to\infty} \lambda_{n,j}(\lambda)=0, \qquad j\geq 3.
\end{equation}
We have proved all claims for $j\geq \frac{n+3}{2}$.
Now let $j<\frac{n+3}{2}$, which is equivalent to $n\geq 2j-2$. Then 
\[
\chi_{n,j}^{\prime}(1)= nj-\frac{n^{2}+3n}{2}<0.
\]
Hence, since $\chi_{n,j}^{\prime\prime}(t)>0$ for all $t>0$, there exists precisely one value 
$\mu_{0}>1$ for which $\chi_{n,j}(\mu_{0})=n$, or equivalently precisely one $\tilde{\lambda}>1/n$ with 
$\lambda_{n,j}(\tilde{\lambda})=1/n$. Again by \eqref{eq:limmes} and continuity,
we must have $\lambda_{n,j}(\lambda)<1/n$ for $\lambda>\tilde{\lambda}$. Moreover, again by
intermediate value theorem either $\lambda_{n,j}(\lambda)>1/n$ for all $\lambda\in{(1/n,\tilde{\lambda})}$
or $\lambda_{n,j}(\lambda)<1/n$ for all $\lambda\in{(1/n,\tilde{\lambda})}$. Suppose conversely
to the claim of the theorem the latter is true.
Recall the implicit equation \eqref{eq:impgl} involving $\lambda_{n}=\lambda$ and $\lambda_{n,j}(\lambda)$. 
Denote
\[
F(x)=\frac{(1+x)^{n+1}}{x}, \qquad G(x,y)=\frac{(1+xy)^{n+1}}{xy},
\]
such that \eqref{eq:impgl} becomes 
$F(\lambda)=G(\lambda^{1-\frac{n+1}{j-1}},\lambda_{n,j}(\lambda)^{1-\frac{n+1}{j-1}})$. 
Proceeding as above, we will show next that for $\lambda$ close to $1/n$ we have
\begin{equation} \label{eq:ratt}
F(\lambda)=G(\lambda^{1-\frac{n+1}{j-1}},\lambda_{n,j}(\lambda)^{\frac{n+1}{j-1}})<
G(\lambda^{1-\frac{n+1}{j-1}},(1/n)^{\frac{n+1}{j-1}}).
\end{equation}
Observe that with $\lambda=\alpha/n$, inequality \eqref{eq:ratt} is equivalent to 
\begin{equation} \label{eq:furiene}
n>\frac{\alpha-\alpha^{\frac{j-n-1}{j-1}}}{\alpha^{\frac{1}{j-1}}-1}.
\end{equation}
Proceeding as above subsequent to \eqref{eq:furien},
we see that for \eqref{eq:furiene} the condition
$\chi_{n,j}^{\prime}(1)>0$ is sufficient. We readily verify that
for $j<\frac{n+3}{2}$ and $\alpha$ sufficiently close to $1$, with a very similar
calculation as in \eqref{eq:kardinal}. Thus we have showed \eqref{eq:ratt}. 
Hence if $\lambda_{n,j}(\lambda)<1/n$ for such $\lambda$,
then by intermediate value theorem of differentiation we must have
\begin{equation} \label{eq:widspruch}
\frac{dG}{dy}(\lambda^{1-\frac{n+1}{j-1}},\eta)>0
\end{equation}
for some pair $(\lambda,\eta)$ with $\lambda\geq 1/n$ and 
$\eta\in(\lambda_{n,j}(\lambda)^{\frac{n+1}{j-1}},(1/n)^{\frac{n+1}{j-1}})$. 
We disprove this. We calculate
\[
\frac{dG(x,y)}{dy}= (nxy-1)(1+xy)^{n}\frac{1}{xy^{2}}.
\]
Hence the sign of the partial derivative of $G$ in \eqref{eq:widspruch}
equals that of $nxy-1$. Our hypothesis yields
\[
n\lambda \eta
\leq n\left(\frac{\alpha}{n}\right)^{1-\frac{n+1}{j-1}}\left(\frac{1}{n}\right)^{\frac{n+1}{j-1}}
= \alpha^{1-\frac{n+1}{j-1}}<1
\]
since $\alpha>1$ and the exponent is negative. Hence $dG(\lambda^{1-\frac{n+1}{j-1}},\eta)/dy<0$
for all $\eta\in(\lambda_{n,j}(\lambda)^{\frac{n+1}{j-1}},(1/n)^{\frac{n+1}{j-1}})$.
This contradicts \eqref{eq:widspruch}.
Hence the hypothesis was wrong and we must have $\lambda_{n,j}(\lambda)>1/n$ for all $\lambda\in{(1/n,\tilde{\lambda})}$.

Finally the fact that $\tilde{\lambda}<n$ follows 
from combination of $\lambda_{n,j}(\tilde{\lambda})=1/n$ and 
$\lambda_{n,j}(\tilde{\lambda})<\tilde{\lambda}^{2-j}\leq \tilde{\lambda}^{-1}<n$ 
for $\tilde{\lambda}>1/n$ and $j\geq 3$, 
see the proof of Proposition~\ref{lagegrund}.
\end{proof}

\subsection{Proofs of Section~\ref{implic}}  \label{4proofs}

We turn towards the proof of Theorem~\ref{reggraph}. We briefly outline a sketch of the proof.
The essential tool for the proof of Theorem~\ref{reggraph} are
special cases of \cite[Theorem 2.2,~2.3 and 2.4]{buschl} 
comprised in Theorem~\ref{zamgfasst} below.

\begin{theorem}[Bugeaud, Schleischitz] \label{zamgfasst}
	Let $n\geq 2$ and $\zeta$ be real transcendental. We have
	\begin{equation} \label{eq:rumnudel}
	\widehat{w}_{n}^{\ast}(\zeta)\leq \frac{nw_{n}(\zeta)}{w_{n}(\zeta)-n+1}.
	\end{equation}
	If $w_{n}(\zeta)>w_{n-1}(\zeta)$ then we have the stronger estimate
	\begin{equation} \label{eq:nudelrum}
	\widehat{w}_{n}(\zeta)\leq \frac{nw_{n}(\zeta)}{w_{n}(\zeta)-n+1}.
	\end{equation}
	If otherwise for $m<n$ we have $w_{m}(\zeta)=w_{n}(\zeta)$, then 
	\begin{equation} \label{eq:mundln}
	\widehat{w}_{n}(\zeta)\leq m+n-1\leq 2n-2.
	\end{equation}
\end{theorem}

Throughout assume Conjecture~\ref{schmsum} holds. Before we proof Theorem~\ref{reggraph},
we want to provide some better numeric results for not too large $n$.
We point out that the functions $\phi_{n}$ are 
increasing. This fact is rather obvious from the definition of the regular graph, 
we omit a rigorous proof.
Let $\widetilde{w}_{n}(\zeta)$ be the solution of the implicit equation
\begin{equation} \label{eq:vier}
\phi_{n}(\widetilde{w}_{n}(\zeta))=
\frac{n\widetilde{w}_{n}(\zeta)}{\widetilde{w}_{n}(\zeta)-n+1}.
\end{equation}
Since $\phi_{n}$ increases whereas the right hand side of \eqref{eq:vier} decreases, 
it follows from \eqref{eq:fischers} and Theorem~\ref{zamgfasst} that
the corresponding value $\phi_{n}(\widetilde{w}_{n}(\zeta))$ is an upper bound for $\widehat{w}_{n}^{\ast}(\zeta)$, and 
in case of $\phi_{n}(\widetilde{w}_{n}(\zeta))\geq 2n-2$ for $\widehat{w}_{n}(\zeta)$ as well. 
For $n\in\{2,3\}$, this procedure leads precisely to the bounds 
$(3+\sqrt{5})/2$ and $3+\sqrt{2}$ in \eqref{eq:tomcat3} and \eqref{eq:tomcat4}, respectively. For $n\geq 4$ not too large, Mathematica can determine
a numerical solution of \eqref{eq:vier}. We provide the implied bounds
\begin{equation} \label{eq:kanada}
\widehat{w}_{4}(\zeta)< 6.2875, \quad 
\widehat{w}_{20}^{\ast}(\zeta)< 37.8787,   \quad
\widehat{w}_{50}^{\ast}(\zeta)< 97.7996.
\end{equation}
Unless $\zeta$ satisfies $w_{n-2}(\zeta)<w_{n-1}(\zeta)=w_{n}(\zeta)$, the above bounds 
for $n\in\{20,50\}$
are valid for $\widehat{w}_{n}(\zeta)$ as well, and we believe the additional condition is 
in fact not necessary. The numeric data
suggests that $2n-\phi_{n}(\widetilde{w}_{n}(\zeta))$ converges to some constant not much larger than the value approximately $0.2004$ we compute with the given bound for $n=50$ above.
In view of this indication, Theorem~\ref{reggraph} is rather satisfactory.
Its proof essentially relies on the above idea, along with asymptotic
estimates for the values $\phi_{n}(\widetilde{w}_{n}(\zeta))$ for large $n$. 
For these estimates we will frequently use the well-known fact that
\begin{equation} \label{eq:euler}
\lim_{n\to\infty} (1+x/n)^{n}=e^{x}
\end{equation}
for real $x$, where the left hand side sequence is monotonic increasing.
We shall also use the variation of \eqref{eq:euler} that for $n\geq 1,\theta>1$ we have 
\begin{equation} \label{eq:insert}
\theta^{-1/(n+1)}<\frac{1}{1+\frac{\log(\theta)}{n+1}}=1-\frac{\log(\theta)}{\log(\theta)+n+1}.
\end{equation}

\begin{proof}[Proof of Theorem~\ref{reggraph}]
First we show \eqref{eq:ohne}. From the assumption of
Conjecture~\ref{schmsum} together with Proposition~\ref{psiprop} and 
\eqref{eq:fischers}, we obtain 
\begin{equation} \label{eq:fritze}
\widehat{w}_{n}^{\ast}(\zeta)\leq \widehat{w}_{n}(\zeta)\leq \phi_{n}(w_{n}(\zeta)). 
\end{equation}
Together with \eqref{eq:rumnudel} we derive
\begin{equation} \label{eq:dieerste}
\widehat{w}_{n}^{\ast}(\zeta)\leq \min\left\{ \frac{nw_{n}(\zeta)}{w_{n}(\zeta)-n+1},\phi_{n}(w_{n}(\zeta))\right\}.
\end{equation}
Let $D\in(1,\Delta)$ be fixed and consider large $n$, in particular $n>3D$. Let
\[
\kappa_{n}:=\frac{(2n-D)(n-1)}{n-D}.
\]
First assume $w_{n}(\zeta)\geq \kappa_{n}$.
Then $nw_{n}(\zeta)/(w_{n}(\zeta)-n+1)\leq 2n-D$ such 
that \eqref{eq:ohne} follows from \eqref{eq:dieerste}. 
Since all $\phi_{n}$ are increasing, it only remains to be shown 
that $\phi_{n}(\kappa_{n})\leq 2n-D$ for large $n$, 
to derive \eqref{eq:dieerste} in case of $n\leq w_{n}(\zeta)<\kappa_{n}$
as well. Hence we may assume $w_{n}(\zeta)=\kappa_{n}$.
It is easy to check
\begin{equation} \label{eq:schaetz}
\kappa_{n}=2n-2+(2-2/n)D+O(1/n)=2n+2D-2+O(1/n).
\end{equation}
In particular $\kappa_{n}=2n+o(n)$. Let
\[
\varphi_{n}(x)=\frac{(x+1)^{n+1}}{x^{n}}=(1+1/x)^{n}(1+x).
\]
With \eqref{eq:euler} we infer
\[
\varphi_{n}(\kappa_{n})= \left(1+\frac{1}{\kappa_{n}}\right)^{n}(\kappa_{n}+1)=
\left(1+\frac{1}{2n+o(n)}\right)^{n}(2n+o(n))
=(2\sqrt{e}+o(1))n.
\]
From \eqref{eq:j1gl1} we further deduce
\[
\varphi_{n}(\widehat{w}_{n,n+1}(\zeta))=\varphi_{n}(w_{n}(\zeta))=\varphi_{n}(\kappa_{n})=(2\sqrt{e}+o(1))n.
\]
We noticed preceding the theorem 
that $\widehat{w}_{n,n+1}(\zeta)\leq n$. Thus if we write $\widehat{w}_{n,n+1}(\zeta)=bn$ then
$b=b(n)\in{[0,1]}$, and again \eqref{eq:euler} yields that $b$ satisfies 
$be^{1/b}=2\sqrt{e}+o(1)$
as $n\to\infty$. This yields $b(n)=\tau+o(1)$ as $n\to\infty$ 
where $\tau\approx 0.5693$ is the solution $y\in{(0,1)}$ to $ye^{1/y}=2\sqrt{e}$. 
Together with \eqref{eq:schaetz} we infer
\begin{equation} \label{eq:ingwer}
\phi(\kappa_{n})=w_{n}(\zeta)\left(\frac{\widehat{w}_{n,n+1}(\zeta)}{w_{n}(\zeta)}\right)^{1/(n+1)}
=(2n+2D-2+o(1))\left(\frac{\tau}{2}+o(1)\right)^{1/(n+1)}.
\end{equation}
Inserting \eqref{eq:insert} with $\theta:=2/\tau\approx 3.5128$ in \eqref{eq:ingwer} yields
\begin{equation} \label{eq:ruppert}
\phi(\kappa_{n})\leq 
(2n+2D-2+o(1))\left(1-\frac{\log(2/\tau+o(1))}{\log(2/\tau+o(1))+n+1}\right).
\end{equation}
One checks that if $D<\Delta=\log(2/\tau)+1$ and $n$ is large then the right hand side of \eqref{eq:ruppert} is smaller 
than $2n-D$. To finish the proof of \eqref{eq:ohne} let $D$ tend to $\Delta$.

Now we show the estimates for $\widehat{w}_{n}(\zeta)$. In case of $w_{n-2}(\zeta)=w_{n}(\zeta)$, from \eqref{eq:mundln}
with $m=n-2$ we derive $\widehat{w}_{n}(\zeta)\leq 2n-3<2n-\Delta$, 
which proves the claim. In case
of $w_{n-1}(\zeta)<w_{n}(\zeta)$, we may apply
\eqref{eq:nudelrum} and obtain the same bounds for 
$\widehat{w}_{n}$ as in \eqref{eq:dieerste}, and can proceed as in the proof of \eqref{eq:ohne}. Hence only possibly in case 
of $w_{n-2}(\zeta)<w_{n-1}(\zeta)=w_{n}(\zeta)$ the bounds may fail, as asserted.
Finally for \eqref{eq:mit} we need precised error terms in dependence of $n$. 
First observe that \eqref{eq:fritze} and Theorem~\ref{zamgfasst} imply
\begin{equation}  \label{eq:japp}
\widehat{w}_{n}(\zeta)\leq \min\left\{ \max\left\{2n-2,\frac{nw_{n}(\zeta)}{w_{n}(\zeta)-n+1}\right\}, 
\phi_{n}(w_{n}(\zeta))\right\}.
\end{equation}
To derive \eqref{eq:mit} we use \eqref{eq:pasta} directly. 
With above argument applied to $D=2$, we see that $w_{n}(\zeta)\geq 2(n-1)^2/(n-2)$ implies $nw_{n}(\zeta)/(w_{n}(\zeta)-n+1)\leq 2n-2$.
Thus \eqref{eq:japp} implies \eqref{eq:mit}. Hence
again since $\phi_{n}$ are monotonic increasing,
it remains to be checked that $\phi_{n}(w)\leq 2n-2$ for $n\geq 10$, 
where $w:=2(n-1)^2/(n-2)$. Let
\[
H(x,y)=x-y+1-\left(\frac{x}{y}\right)^{n}.
\]
Recall $(w_{n}(\zeta),\phi_{n}(w_{n}(\zeta))=(w_{n}(\zeta),\widehat{w}_{n}(\zeta))$ 
satisfy \eqref{eq:pasta}. 
In particular $H(w,\phi(w))=0$ or $\phi_{n}(w)$ is the solution $y_{0}<w$ of
\[
H(w,y_{0})=\frac{2(n-1)^2}{n-2}-y_{0}+1+\left(\frac{2(n-1)^2}{(n-2)y_{0}}\right)^{n}=0.
\]
Some elementary calculation shows
\[
H(w,2n-2)=\frac{2}{n-2}+3-\left(\frac{n-1}{n-2}\right)^{n}=\frac{2}{n-2}+3-
\left(1+\frac{1}{n-2}\right)^{n-2}\left(1+\frac{1}{n-2}\right)^{2}.
\]
Together with \eqref{eq:euler} and some computation for small $n$, the right hand side 
can be easily checked to be positive for $n\geq 10$.
On the other hand we have
\[
\frac{dH}{dy}(w,y)=-1+nw^{n}y^{-n-1},
\]
which is positive for any $y<\phi_{n}(w)$ by \eqref{eq:glueck}. Thus indeed the root 
$y_{0}=\phi_{n}(w)$ of $H(w,y_{0})=0$ must be smaller than $2n-2$. 
This finishes the proof.
\end{proof}

\subsection{Proofs of Section~\ref{andere}}

In the proof of Theorem~\ref{zwischen} we will apply the transference inequality
\begin{equation} \label{eq:ogerman}
\frac{\widehat{w}_{n}(\underline{\zeta})}{\widehat{w}_{n}(\underline{\zeta})-n+1} 
\geq \widehat{\lambda}_{n}(\underline{\zeta})\geq 
\frac{\widehat{w}_{n}(\underline{\zeta})-1}{(n-1)\widehat{w}_{n}(\underline{\zeta})}
\end{equation}
due to German~\cite{german}, valid for all $n\geq 1$ and $\underline{\zeta}\in{\mathbb{R}^{n}}$ that are
$\mathbb{Q}$-linearly independent together with $\{1\}$.

\begin{proof}[Proof of Theorem~\ref{zwischen}]
Too keep the notation simple we restrict to vectors $(\zeta,\zeta^{2},\ldots,\zeta^{n})$, 
the proof
can be readily generalized to linear forms in arbitrary $\underline{\zeta}$.
Let $\epsilon>0$. By definition of $\widehat{w}_{n}(\zeta)$,
for any sufficiently large $k$ we have
\begin{equation} \label{eq:zweit}
\vert P_{k+1}(\zeta)\vert<\vert P_{k}(\zeta)\vert<H(P_{k+1})^{-\widehat{w}_{n}(\zeta)+\epsilon}.
\end{equation}

On the other hand, it follows from the definitions of $w_{n}(\zeta)$ and $\widehat{w}_{n}(\zeta)$
that for large $l$ two successive best approximations $P_{l},P_{l+1}$ 
satisfy
$\log H(P_{l+1})/\log H(P_{l})\leq w_{n}(\zeta)/\widehat{w}_{n}(\zeta)+\epsilon$, 
or equivalently
$\log H(P_{l})/\log H(P_{l+1})\geq \widehat{w}_{n}(\zeta)/w_{n}(\zeta)-\tilde{\epsilon}$, where
$\tilde{\epsilon}$ tends to $0$ as $\epsilon$ does.
This same argument applied repeatedly for $l$ from $k+1$ to $k+i-2$ shows that 
\begin{equation} \label{eq:bester}
\frac{\log H(P_{k+1})}{\log H(P_{k+i-1})}\geq \left(\frac{\widehat{w}_{n}(\zeta)}{w_{n}(\zeta)}\right)^{i-2}-
\tilde{\epsilon}_{1},
\end{equation}
for some $\tilde{\epsilon}_{1}$ which depends on $\epsilon$ and tends to $0$ as $\epsilon$ tends to $0$.
Combination of \eqref{eq:zweit} and \eqref{eq:bester} yields
\[
-\frac{\log \vert P_{k}(\zeta)\vert}{\log H(P_{k+i-1})}=-\frac{\log \vert P_{k}(\zeta)\vert}{\log H(P_{k+1})}\cdot
\frac{\log H(P_{k+1})}{\log H(P_{k+i-1})}
\geq (\widehat{w}_{n}(\zeta)-\epsilon)
\left(\left(\frac{\widehat{w}_{n}(\zeta)}{w_{n}(\zeta)}\right)^{i-2}-
\tilde{\epsilon}_{1}\right).
\]
Since $\vert P_{k}(\zeta)\vert > \vert P_{k+1}(\zeta)\vert > \cdots > \vert P_{k+n}(\zeta)\vert$ we infer 
that 
\[
-\frac{\log \vert P_{k+j}(\zeta)\vert}{\log H(P_{k+i-1})}\geq 
\widehat{w}_{n}(\zeta)\left(\frac{\widehat{w}_{n}(\zeta)}{w_{n}(\zeta)}\right)^{n-1}+\tilde{\epsilon}_{2},
\qquad 0\leq j\leq i-1,
\]
for some $\tilde{\epsilon}_{2}$ which again depends on $\epsilon$ and tends to $0$ as $\epsilon$ does.
Moreover by our assumption we can find arbitrarily large $k$ such that the polynomials $P_{k},P_{k+1},\ldots,P_{k+n}$
are linearly independent. Hence and since $H(P_{k+n})\geq H(P_{k+j})$ for $0\leq j\leq n$, 
we obtain \eqref{eq:heute} as we may take $\epsilon$ arbitrarily small. The estimate \eqref{eq:fastfertig}
is unconditioned since for $i=3$ Conjecture~\ref{bestapp} is unconditioned, see Remark~\ref{markenrehe}.

Finally \eqref{eq:99lb} follows from \eqref{eq:heute} with $i=n+1$ combined with 
\[
w_{n,n+1}(\zeta)=\frac{1}{\widehat{\lambda}_{n}(\zeta)}\leq \frac{(n-1)\widehat{w}_{n}(\zeta)}{\widehat{w}_{n}(\zeta)-1},
\]
by elementary rearrangements.
The right above inequality is obtained from \eqref{eq:ogerman} by taking reciprocals.
The dual estimates for the constants $\lambda_{n,j},\widehat{\lambda}_{n,j}$ are obtained very similarly,
where for \eqref{eq:feierabend} we applied
\[
\lambda_{n,n+1}(\zeta)=\frac{1}{\widehat{w}_{n}(\zeta)}\leq \frac{1-\widehat{\lambda}_{n}(\zeta)}{n-1}, 
\]
where again we used \eqref{eq:ogerman}.
\end{proof}

\begin{proof}[Proof of Theorem~\ref{anderer}]
	By assumption and Theorem~\ref{zwischen} inequality
	\eqref{eq:99lb} holds, which is stronger than \eqref{eq:glueck}. 
	As mentioned at the end of Section~\ref{implic} this 
	estimation in turn implies the claim \eqref{eq:wegweg}.
\end{proof}

\end{document}